\documentclass[12pt]{amsart}

\usepackage{graphicx}

\usepackage{dsfont}

\newtheorem{theorem}{Theorem}
\newtheorem{lemma}{Lemma}
\newtheorem{proposition}{Proposition}

\theoremstyle{definition}
\newtheorem{definition}{Definition}
\newtheorem{example}{Example}
\newtheorem{claim}{Claim}
\theoremstyle{remark}
\newtheorem{remark}{Remark}

\newtheorem*{acknowledgement}{Acknowledgement}

\begin{document}


\title{Fourier Frames for the Cantor-4 Set}
\author{Gabriel Picioroaga}
\address{Department of Mathematical Sciences, 414 E. Clark St., University of South Dakota, Vermillion, SD 57069}
\email{gabriel.picioroaga@usd.edu}
\author{Eric S. Weber}
\address{Department of Mathematics, Iowa State University, 396 Carver Hall, Ames, IA 50011}
\email{esweber@iastate.edu}
\date{\today}
\begin{abstract}
The measure supported on the Cantor-4 set constructed by Jorgensen-Pedersen is known to have a Fourier basis, i.e. that it possess a sequence of exponentials which form an orthonormal basis.  We construct Fourier frames for this measure via a dilation theory type construction.  We expand the Cantor-4 set to a 2 dimensional fractal which admits a representation of a Cuntz algebra.  Using the action of this algebra, an orthonormal set is generated on the larger fractal, which is then projected onto the Cantor-4 set to produce a Fourier frame.
\end{abstract}
\maketitle

Jorgensen and Pedersen \cite{JP98} demonstrated that there exist singular measures $\nu$ which are spectral--that is, they possess a sequence of exponential functions which form an orthonormal basis in $L^2(\nu)$.  The canonical example of such a singular and spectral measure is the uniform measure on the Cantor 4-set defined as follows:  
\[ C_{4} = \{ x \in [0,1] : x = \sum_{k=1}^{\infty} \dfrac{a_{k}}{4^{k}}, \ a_{k} \in \{0,2\} \}. \]
This is analogous to the standard middle third Cantor set where $4^{k}$ replaces $3^{k}$.  The set $C_{4}$ can also be described as the attractor set of the following iterated function system on $\mathbb{R}$:
\[ \tau_{0}(x) = \dfrac{x}{4}, \qquad \tau_{2}(x) = \dfrac{x + 2}{4}. \]
The uniform measure on the set $C_{4}$ then is the unique probability measure $\mu_{4}$ which is invariant under this iterated function system:
\[ \int f(x) d \mu_{4}(x) = \dfrac{1}{2} \left( \int f( \tau_{0} (x) ) d \mu_{4}(x) + \int f( \tau_{2} (x) ) d \mu_{4}(x) \right) \]
for all $f \in C(\mathbb{R})$, see \cite{Hut81} for details.  The standard spectrum for $\mu_{4}$ is $\Gamma_{4} = \{ \sum_{n=0}^{N} l_{n} 4^{n} : l_{n} \in \{0,1\} \}$, though there are many spectra \cite{DHS09a,DHL13a}.

Remarkably, Jorgensen and Pedersen prove that the uniform measure $\mu_{3}$ on the standard middle third Cantor set is not spectral.  Indeed, there are no three mutually orthogonal exponentials in $L^2(\mu_{3})$.  Thus, there has been much attention on whether there exists a Fourier frame for $L^2(\mu_{3})$--the problem is still unresolved, but see \cite{DHW11a,DHW14a} for progress in this regard.  In this paper, we will construct Fourier frames for $L^2(\mu_{4})$ using a dilation theory type argument.  The motivation is whether the construction we demonstrate here for $\mu_{4}$ will be applicable to $\mu_{3}$.  Fourier frames for $\mu_{4}$ were constructed in \cite{DHW14a} using a duality type construction.

A frame for a Hilbert space $H$ is a sequence $\{ x_{n} \}_{n \in I} \subset H$ such that there exists constants $A, B > 0$ such that for all $v \in H$,
\[ A \| v \|^2 \leq \sum_{n \in I} | \langle v , x_{n} \rangle |^2 \leq B \| v \|^2. \]
The largest $A$ and smallest $B$ which satisfy these inequalities are called the frame bounds.  The frame is called a Parseval frame if both frame bounds are $1$.  The sequence $\{x_{n} \}_{n \in I}$ is a Bessel sequence if there exists a constant $B$ which satisfies the second inequality, whether or not the first inequality holds; $B$ is called the Bessel bound.  A Fourier frame for $L^2(\mu_{4})$ is a sequence of frequencies $\{ \lambda_{n} \}_{n \in I} \subset \mathbb{R}$ together with a sequence of ``weights'' $\{ d_{n} \}_{n \in I} \subset \mathbb{C}$ such that $x_{n} = d_{n} e^{2 \pi i \lambda_{n} x}$ is a frame.  Fourier frames (unweighted) for Lebesgue measure were introduced by Duffin and Schaffer \cite{DS52a}, see also Ortega-Cerda and Seip \cite{OCS02a}.

It was proven in \cite{HL00a} that a frame for a Hilbert space can be dilated to a Riesz basis for a bigger space, that is to say, that any frame is the image under a projection of a Riesz basis.  Moreover, a Parseval frame is the image of an orthonormal basis under a projection.  This result is now known to be a consequence of the Naimark dilation theory.   This will be our recipe for constructing a Fourier frame: constructing a basis in a bigger space and then projecting onto a subspace.  We require the following result along these lines \cite{A95a}:

\begin{lemma} \label{L:frame-proj}
Let $H$ be a Hilbert space, $V,K$ closed subspaces, and let $P_{V}$ be the projection onto $V$.  If $\{ x_{n} \}_{n \in I}$ is a frame in $K$ with frame bounds $A,B$, then:
\begin{enumerate}
\item $\{ P_{V} x_{n} \}_{n \in I}$ is a Bessel sequence in $V$ with Bessel bound no greater than $B$;
\item if the projection $P_{V} : K \to V$ is onto, then $\{P_{V} x_{n} \}_{n \in I}$ is a frame in $V$;
\item if $V \subset K$, then then $\{P_{V} x_{n} \}_{n \in I}$ is a frame in $V$ with frame bounds between $A$ and $B$.
\end{enumerate}
\end{lemma}

Note that if $V \subset K$ and $\{x_{n} \}_{n \in I}$ is a Parseval frame for $K$, then $\{ P_V x_{n} \}_{n \in I}$ is a Parseval frame for $V$.  In the second item above, it is possible that the lower frame bound for $\{ P_{V} x_{n} \}$ is smaller than $A$, but the upper frame bound is still no greater than $B$.

The foundation of our construction is a dilation theory type argument.  Our first step, described in Section \ref{Sec:dilation}, is to consider the  fractal like set $C_{4} \times [0,1]$, which we will view in terms of an iterated function system.  This IFS will give rise to a representation of the Cuntz algebra $\mathcal{O}_{4}$ on $L^2(\mu_{4} \times \lambda)$ since $\mu_{4} \times \lambda$ is the invariant measure under the IFS.  Then in Section \ref{Sec:ortho}, we will generate via the action of $\mathcal{O}_{4}$ an orthonormal set in $L^2(\mu_{4} \times \lambda)$ whose vectors have a particular structure.  In Section \ref{Sec:proj}, we consider a subspace $V$ of $L^2(\mu_{4} \times \lambda)$ which can be naturally identified with $L^2(\mu_{4})$, and then project the orthonormal set onto $V$ to, ultimately, obtain a frame.  Of paramount importance will be whether the orthonormal set generated by $\mathcal{O}_{4}$ spans the subspace $V$ so that the projection yields a Parseval frame.  Section \ref{Sec:const} demonstrates concrete constructions in which this occurs, and identifies all possible Fourier frames that can be constructed using this method.

We note here that there may be Fourier frames for $L^2(\mu_{4})$ which cannot be constructed in this manner, but we are unaware of such an example.

\section{Dilation of the Cantor-4 Set} \label{Sec:dilation}

We wish to construct a Hilbert space $H$ which contains $L^2(\mu_{4})$ as a subspace in a natural way.  We will do this by making the fractal $C_{4}$ bigger as follows.  We begin with an iterated function system on $\mathbb{R}^2$ given by:
\[ \Upsilon_{0} (x,y) = ( \frac{x}{4}, \frac{y}{2} ), \ \Upsilon_{1} (x,y) = ( \frac{x+2}{4}, \frac{y}{2} ), \  \Upsilon_{2} (x,y) = ( \frac{x}{4}, \frac{y+1}{2} ), \  \Upsilon_{3} (x,y) = ( \frac{x+2}{4}, \frac{y+1}{2} ). \]
As these are contractions on $\mathbb{R}^2$, there exists a compact attractor set, which is readily verified to be $C_{4} \times [0,1]$.  Likewise, by Hutchinson \cite{Hut81}, there exists an invariant probability measure supported on $C_{4} \times [0,1]$; it is readily verified that this invariant measure is $\mu_{4} \times \lambda$, where $\lambda$ denotes the Lebesgue measure restricted to $[0,1]$.  Thus, for every continuous function $f : \mathbb{R}^2 \to \mathbb{C}$,
\begin{multline} \label{Eq:refine}
\int f(x,y) \ d (\mu_{4} \times \lambda) = \dfrac{1}{4} \left( \int f(\frac{x}{4},\frac{y}{2}) \ d (\mu_{4} \times \lambda) + \int f(\frac{x+2}{4},\frac{y}{2}) \ d (\mu_{4} \times \lambda) \right. \\ 
\left. + \int f(\frac{x}{4},\frac{y+1}{2}) \ d (\mu_{4} \times \lambda) + \int f(\frac{x+2}{4},\frac{y+1}{2}) \ d (\mu_{4} \times \lambda) \right). 
\end{multline}
The iterated function system $\Upsilon_{j}$ has a left inverse on $C_{4} \times [0,1]$, given by 
\[ R: C_{4} \times [0,1] \to C_{4} \times [0,1] : (x,y) \mapsto  (4x, 2y) \mod 1, \]
so that $R \circ \Upsilon_{j} (x,y) = (x,y)$ for $j=0,1,2,3$.

We will use the iterated function system to define an action of the Cuntz algebra $\mathcal{O}_{4}$ on $L^2(\mu_{4} \times \lambda)$.  To do so, we  choose filters
\begin{align*}
m_{0}(x,y) &= H_{0}(x,y) \\
m_{1}(x,y) &= e^{2 \pi i x} H_{1}(x,y) \\
m_{2}(x,y) &= e^{4 \pi i x} H_{2}(x,y) \\
m_{3}(x,y) &= e^{6 \pi i x} H_{3}(x,y)
\end{align*}
where
\[ H_{j}(x,y) = \sum_{k=0}^{3} a_{jk} \chi_{\Upsilon_{k}(C_{4} \times [0,1])}(x,y) \]
for some choice of scalar coefficients $a_{jk}$.  In order to obtain a representation of $\mathcal{O}_{4}$ on $L^2(\mu_{4} \times \lambda)$, we require that the above filters satisfy the matrix equation $\mathcal{M}^{*}(x,y) \mathcal{M}(x,y) = I$ for $\mu_{4} \times \lambda$ almost every $(x,y)$, where
\[ \mathcal{M}(x,y) =
\begin{pmatrix}
m_{0}(\Upsilon_{0}(x,y))  & m_{0}(\Upsilon_{1}(x,y)) & m_{0}(\Upsilon_{2}(x,y)) & m_{0}(\Upsilon_{3}(x,y)) \\
m_{1}(\Upsilon_{0}(x,y))  & m_{1}(\Upsilon_{1}(x,y)) & m_{1}(\Upsilon_{2}(x,y)) & m_{1}(\Upsilon_{3}(x,y)) \\
m_{2}(\Upsilon_{0}(x,y))  & m_{2}(\Upsilon_{1}(x,y)) & m_{2}(\Upsilon_{2}(x,y)) & m_{2}(\Upsilon_{3}(x,y)) \\
m_{3}(\Upsilon_{0}(x,y))  & m_{3}(\Upsilon_{1}(x,y)) & m_{3}(\Upsilon_{2}(x,y)) & m_{3}(\Upsilon_{3}(x,y))
\end{pmatrix}
\]
For our choice of filters, the matrix $\mathcal{M}$ becomes 
\[ \mathcal{M}(x,y) =
\left(
\begin{array} {rrrr}
a_{00} & a_{01} & a_{02} & a_{03} \\
e^{\pi i x/2} a_{10} & -e^{\pi i x/2} a_{11} & e^{\pi i x/2} a_{12} & -e^{\pi i x/2} a_{13} \\
e^{\pi i x} a_{20} & e^{\pi i x} a_{21} & e^{\pi i x} a_{22} & e^{\pi i x} a_{23} \\
e^{3 \pi i x/2} a_{30} & -e^{3 \pi i x/2}a_{31} & e^{3 \pi i x/2}a_{32} & -e^{3 \pi i x/2}a_{33}
\end{array}
\right),
\]
which is unitary if and only if the matrix
\[ H =
\left(
\begin{array} {rrrr}
a_{00} & a_{01} & a_{02} & a_{03} \\
a_{10} & -a_{11} &  a_{12} & -a_{13} \\
a_{20} & a_{21} &  a_{22} & a_{23} \\
a_{30} & -a_{31} & a_{32} & -a_{33}
\end{array}
\right)
\]
is unitary.  For the remainder of this section, we assume that $H$ is unitary.

\begin{lemma}
The operator $S_{j} : L^2(\mu_{4} \times \lambda) \to L^2(\mu_{4} \times \lambda)$ given by
\[ [S_{j}f] (x,y) = \sqrt{4} m_{j}(x,y) f(R(x,y)) \]
is an isometry.
\end{lemma}

\begin{proof}
We calculate:
\begin{align*}
\| S_{j}f \|^2 &= \int |\sqrt{4} m_{j}(x,y) f(R(x,y))|^2 \ d(\mu_{4} \times \lambda) \\
&= \dfrac{1}{4} \sum_{k=0}^{3} \int 4 | m_{j}(\Upsilon_{k}(x,y)) f(R(\Upsilon_{k}(x,y))) |^2 \ d(\mu_{4} \times \lambda) \\
&=  \int \left( \sum_{k=0}^{3} | m_{j}(\Upsilon_{k}(x,y))|^2 \right) |f(x,y)|^2 \ d(\mu_{4} \times \lambda).
\end{align*}
We used Equation (\ref{Eq:refine}) in the second line.  The sum in the integral is the square of the Euclidean norm of the $j$-th row of the matrix $\mathcal{M}$, which is unitary.  Hence, the sum is $1$, so the integral is $\| f \|^2$, as required.
\end{proof}

\begin{lemma}
The adjoint is given by
\[ [S_{j}^{*} f] (x,y) = \dfrac{1}{2} \sum_{k=0}^{3} \overline{m_{j} (\Upsilon_{k}(x,y)) } f(\Upsilon_{j}(x,y)). \]
\end{lemma}

\begin{proof}
Let $f,g \in L^2(\mu_{4} \times \lambda)$.  We calculate
\begin{align*}
\langle S_{j}f, g \rangle &= \int \sqrt{4} m_{j}(x,y) f(R(x,y)) \overline{ g(x,y) } \ d(\mu_{4} \times \lambda) \\
&= \dfrac{1}{4} \sum_{k=0}^{3} \int \sqrt{4} m_{j}(\Upsilon_{k}(x,y)) f(R(\Upsilon_{k}(x,y))) \overline{ g(\Upsilon_{k} (x,y)) } \ d(\mu_{4} \times \lambda) \\
&=  \int f(x,y) \overline{ \left( \dfrac{1}{2} \sum_{k=0}^{3}  \overline{m_{j}(\Upsilon_{k}(x,y))} g(\Upsilon_{k} (x,y)) \right) } \ d(\mu_{4} \times \lambda)
\end{align*}
where we use Equation (\ref{Eq:refine}) and the fact that $R$ is a left inverse of $\Upsilon_{k}$.
\end{proof}

\begin{lemma}
The isometries $S_{j}$ satisfy the Cuntz relations:
\[ S_{j}^{*} S_{k} = \delta_{jk} I, \qquad \sum_{k=0}^{3} S_{k} S_{k}^{*} = I. \]
\end{lemma}

\begin{proof}
We consider the orthogonality relation first.  Let $f \in L^2(\mu_{4} \times \lambda)$.  We calculate:
\begin{align*}
[S_{j}^{*} S_{k} f](x,y) &= \dfrac{1}{2} \sum_{\ell = 0}^{3} \overline{ m_{j}(\Upsilon_{\ell}(x,y))} [S_{k} f]( \Upsilon_{\ell}(x,y) ) \\
&= \dfrac{1}{2} \sum_{\ell = 0}^{3} \overline{ m_{j}(\Upsilon_{\ell}(x,y))} \sqrt{4} m_{k}(\Upsilon_{\ell}(x,y)) f( R(\Upsilon_{\ell}(x,y) ) ) \\
&= \left( \sum_{\ell = 0}^{3} \overline{ m_{j}(\Upsilon_{\ell}(x,y))} m_{k}(\Upsilon_{\ell}(x,y)) \right) f(x,y). 
\end{align*}
Note that the sum is the scalar product of the $k$-th row with the $j$-th row of the matrix $\mathcal{M}$, which is unitary.  Hence, the sum is $\delta_{jk}$ as required.

Now for the identity relation, let $f,g \in L^2(\mu_{4} \times \lambda)$.  We calculate:
{\allowdisplaybreaks
\begin{align}
\langle \sum_{k=0}^{3} &S_{k} S_{k}^{*} f, g \rangle = \sum_{k=0}^{3} \langle S_{k}^{*} f, S_{k}^{*} g \rangle \notag \\
&= \sum_{k=0}^{3} \int \left( \dfrac{1}{2} \sum_{\ell = 0}^{3} \overline{ m_{k}(\Upsilon_{\ell} (x,y)) } f(\Upsilon_{\ell}(x,y)) \right) \left( \overline{ \dfrac{1}{2} \sum_{n = 0}^{3} \overline{ m_{k}(\Upsilon_{n} (x,y)) } g(\Upsilon_{n}(x,y))} \right) \ d(\mu_{4} \times \lambda) \notag \\
&= \sum_{\ell=0}^{3} \sum_{n=0}^{3} \dfrac{1}{4} \int \left( \sum_{k = 0}^{3} \overline{ m_{k}(\Upsilon_{\ell} (x,y)) }m_{k}(\Upsilon_{n} (x,y)) \right)  f(\Upsilon_{\ell}(x,y)) \overline{ g(\Upsilon_{n}(x,y))}  \ d(\mu_{4} \times \lambda) \notag \\
&= \dfrac{1}{4} \sum_{n=0}^{3} \int f(\Upsilon_{n}(x,y)) \overline{ g(\Upsilon_{n}(x,y))}  \ d(\mu_{4} \times \lambda) \notag \\
&= \int f(x,y) \overline{ g(x,y)}  \ d(\mu_{4} \times \lambda) \notag \\
&= \langle f, g \rangle. \notag
\end{align}
}
Note that the sum over $k$ in the third line is the scalar product of the $\ell$-th column with the $n$-th column of $\mathcal{M}$, so the sum collapses to $\delta_{\ell n}$.  The sum on $n$ in the fourth line collapses by Equation (\ref{Eq:refine}).
\end{proof}

\section{Orthonormal Sets in $L^2(\mu_{4} \times \lambda)$} \label{Sec:ortho}

Since the isometries $S_{j}$ satisfy the Cuntz relations, we can use them to generate orthonormal sets in the space $L^2(\mu_{4} \times \lambda)$.  We do so by having the isometries act on a generating vector.  We consider words in the alphabet $\{0,1,2,3\}$; let $W_{4}$ denote the set of all such words.  For a word $\omega = j_{K} j_{K-1} \dots j_{1}$, we denote by $| \omega | = K$ the length of the word, and define
\[ S_{\omega} f = S_{j_{K}} S_{j_{K-1}} \dots S_{j_{1}} f. \]
\begin{definition}
Let
\[ X_{4} = \{ \omega \in W_{4}: | \omega | = 1 \} \cup \{ \omega \in W_{4} : | \omega | \geq 2, \ j_{1} \neq 0 \}. \]
\end{definition}
For convenience, we allow the empty word $\omega_{\emptyset}$ with length $0$, and define $S_{\omega_{\emptyset}} = I$, the identity.

\begin{lemma} \label{L:orthonormal}
Suppose $f \in L^2(\mu_{4} \times \lambda)$ with $\| f\| = 1$, and that $S_{0} f = f$.  Then,
\[ \{ S_{\omega} f : \omega \in X_{4} \} \]
is an orthonormal set.
\end{lemma}

\begin{proof}
Suppose $\omega, \omega' \in X_{4}$ with $\omega \neq \omega'$.  First consider $| \omega | = | \omega' |$, with $\omega = j_{K} \dots j_{1}$ and $\omega' = i_{K} \dots i_{1}$.  Suppose that $\ell$ is the largest index such that $j_{\ell} \neq i_{\ell}$.  Then we have
\[ \langle S_{\omega} f , S_{\omega '} f \rangle = \langle S_{j_{\ell}} \dots S_{j_{1}} f , S_{i_{\ell}} \dots S_{i_{1}} f \rangle =
\langle S_{i_{\ell}}^{*} S_{j_{\ell}} \dots S_{j_{1}} f , S_{i_{\ell - 1}} \dots S_{i_{1}} f \rangle = 0
\]
by the orthogonality condition of the Cuntz relations.

Now, if $K = | \omega | > | \omega' | = M$, with $\omega' = i_{M} \dots i_{1}$, we define the word $\rho = i_{M} \dots i_{1} 0 \dots 0$ so that $| \rho | = K$.  Note that $\rho \notin X_{4}$ so $\omega \neq \rho$.  Note further that $S_{\omega'} f = S_{\rho} f$.  Thus, by a similar argument to that above, we have
\[ \langle S_{\omega} f , S_{\omega'} f \rangle = 0. \]
\end{proof}

\begin{remark}
The set $\{ S_{\omega} f : \omega \in X_{4} \}$ need not be complete.  We will provide an example of this in Example \ref{Ex:incomplete} in Section \ref{Sec:const}.
\end{remark}

Our goal is to project the set $\{ S_{\omega} f : \omega \in X_{4} \}$ onto some subspace $V$ of $L^2(\mu_{4} \times \lambda)$ to obtain a frame.  To that end, we need to know when the projection $\{ P_{V} S_{\omega} f : \omega \in X_{4} \}$ is a frame, which by Lemma \ref{L:frame-proj} requires the projection $P_{V} : K \to V$ to be onto, where $K$ is the subspace spanned by $\{ S_{\omega} f : \omega \in X_{4} \}$.  The tool we will use is the following result, which is a minor adaptation of a result from \cite{DPS14a}; we will not use this result directly, but will use all of the critical components.

\begin{theorem}\label{diez-subsp}
Let $\mathcal{H}$ be a Hilbert space, $\mathcal{K} \subset \mathcal{H}$ a closed subspace, and $(S_i)_{i=0}^{N-1}$ be a representation of the Cuntz algebra $\mathcal{O}_N$. Let $\mathcal{E}$ be an orthonormal set in $\mathcal{H}$ and $f:X\rightarrow \mathcal{K}$ a norm continuous function on a topological space $X$ with the following properties:

\begin{enumerate}

\item[i)] $\mathcal{E}=\cup_{i=0}^{N-1} S_i\mathcal{E}$ where the union is disjoint.

\item[ii)] $\overline{span}\{ f(t): t\in X\}=\mathcal{K}$ and $\|f(t)\|=1$, for all $t\in X$.

\item[iii)] There exist functions $\mathfrak{m}_i: X\rightarrow\mathbb{C}$, $g_i:X\rightarrow X$, $i=0,\dots, N-1$ such that
\begin{equation} \label{ceq1s}
S_i^*f(t)=\mathfrak{m}_i(t)f(g_i(t)), \quad t\in X.
\end{equation}

\item[iv)] There exist $c_0\in X$ such that $f(c_0)\in \overline{span} \mathcal{E}$.

\item[v)] The only function $h\in\mathcal{C}(X)$ with $h\geq 0$, $h(c)=1$, $\forall$ $c\in\{ x\in X:f(x)\in \overline{span}\mathcal{E}\}$, and
\begin{equation} \label{ceq2s}
h(t)=\sum_{i=0}^{N-1}\left| \mathfrak{m}_i(t) \right|^2 h(g_i(t)),\quad t\in X
\end{equation}
are the constant functions.

\end{enumerate}
Then $\mathcal{K} \subset \overline{span}\mathcal{E}$.
\end{theorem}

\section{The Projection} \label{Sec:proj}

Recall the definition of the filters $m_{j}(x,y) = e^{2 \pi i j x} H_{j} (x,y)$ from Section \ref{Sec:dilation}.  We choose the filter coefficients $a_{j k}$ so that the matrix $H$ is unitary.  We place the additional constraint that 
\[ a_{00} = a_{01} = a_{02} = a_{03} = \dfrac{1}{2}, \] 
so that $S_{0} \mathds{1} = \mathds{1}$, where $\mathds{1}$ the function in $L^2(\mu_{4} \times \lambda)$ which is identically $1$.  As $S_{0} \mathds{1} = \mathds{1}$, by Lemma \ref{L:orthonormal}, the set $\{ S_{\omega} \mathds{1} : \omega \in X_{4} \}$ is orthonormal.   Moreover, we place the additional constraint that for every $j$, $a_{j0} + a_{j2} = a_{j1} + a_{j3}$, which will be required for our calculation of the projection.  

\begin{definition}
We define the subspace $V = \{ f \in L^2(\mu_{4} \times \lambda) : f(x,y) = g(x) \chi_{[0,1]}(y), \ g \in L^2(\mu_{4}) \}$.  Note that the subspace $V$ can be identified with $L^2(\mu_{4})$ via the isometric isomorphism $g \mapsto g(x) \chi_{[0,1]}(y)$.  We will suppress the $y$ variable in the future.
\end{definition}

\begin{definition}
We define a function $c : X_{4} \to \mathbb{N}_{0}$ as follows:  for a word $\omega = j_{K} j_{K-1} \dots j_{1}$,
\[ c(\omega) = \sum_{k=1}^{K} j_{k} 4^{K-k}. \]
Here $\mathbb{N}_{0} = \mathbb{N} \cup \{0\}$. It is readily verified that $c$ is a bijection.
\end{definition}

\begin{lemma} \label{L:Sword}
For a word $\omega = j_{K} j_{K-1} \dots j_{1}$,
\[ S_{\omega} \mathds{1} = e^{2 \pi i c(\omega) x} \left( \prod_{k=1}^{K} 2 H_{j_{k}}(R^{K-k}(x,y)) \right). \]
\end{lemma}

\begin{proof}
We proceed by induction on the length of the word $\omega$.  The equality is readily verified for $| \omega | =1$.  Let $\omega_{0} = j_{K-1} j_{n-2} \dots j_{1}$.  We have
\begin{align*}
S_{\omega} \mathds{1} &= S_{j_{K}} S_{\omega_{0}} \mathds{1} \\
&= S_{j_{K}} \left[ e^{2 \pi i c(\omega_{0}) x } \left( \prod_{k=1}^{K-1} 2 H_{j_{k}}(R^{K-1-k}(x,y)) \right) \right] \\
&= 2 e^{2 \pi i \lambda_{j_{K}} x} H_{j_{K}}(x,y) e^{2 \pi i c(\omega_{0}) \cdot 4x } \left( \prod_{k=1}^{K-1} H_{j_{k}}(R^{K-k}(x,y)) \right) \\
&= 2 e^{2 \pi i (\lambda_{j_{K}} + 4c(\omega_{0}) ) x} H_{j_{K}}(R^{K-K}(x,y)) \left( \prod_{k=1}^{K-1} 2 H_{j_{k}}(R^{K-k}(x,y)) \right) \\
&= e^{2 \pi i c(\omega) x} \left( \prod_{k=1}^{K} 2 H_{j_{k}}(R^{K-k}(x,y)) \right).
\end{align*}
The last line above is justified by the following calculation:
\begin{align*}
\lambda_{j_{K}} + 4 c(\omega_{0})  &=  \lambda_{j_{K}} + 4 \left( \sum_{k=1}^{K-1} \lambda_{j_{k}}4^{K-1-k} \right)  \\
&= \lambda_{j_{K}} 4^{K-K} +  \sum_{k=1}^{K-1} \lambda_{j_{k}} 4^{K-k} \\
&= \sum_{k=1}^{K} \lambda_{j_{k}} 4^{K-k} \\
&= c(\omega).
\end{align*}
\end{proof}

We wish to project the vectors $S_{\omega} \mathds{1}$ onto the subspace $V$.  The following lemma calculates that projection, where $P_{V}$ denotes the projection onto the subspace $V$.

\begin{lemma} \label{L:project}
If $f(x,y) = g(x)h(x,y)$ with $g \in L^{2}(\mu_{4})$ and $h \in L^{\infty}(\mu_{4} \times \lambda)$, then
\[ [P_{V} f](x,y) =  g(x) G(x) \]
where
$G(x) = \int_{[0,1]} h(x,y) d \lambda(y)$.
\end{lemma}

\begin{proof}
We verify that for every $F(x) \in L^2(\mu_{4})$, $f(x,y) - g(x) G(x)$ is orthogonal to $F(x)$.  We calculate utilizing Fubini's theorem:
\begin{align*}
\langle f - gG , F \rangle &= \int \int g(x) h(x,y) \overline{F(x)} \ d (\mu_{4} \times \lambda) - \int \int g(x)G(x)\overline{F(x)} \ d (\mu_{4} \times \lambda) \\
&=  \int_{C_{4}} g(x)\overline{F(x)} \left(\int_{[0,1]} h(x,y) - G(x) \ d \lambda(y) \right) \ d \mu_{4}(x) \\
&=  \int_{C_{4}} g(x)\overline{F(x)} \left( G(x) - G(x)  \right) \ d \mu_{4}(x) \\
&= 0.
\end{align*}
\end{proof}

For the purposes of the following lemma, $\alpha x$ and $\beta y$ are understood to be modulo $1$.

\begin{lemma} \label{L:integrate}
For any word $\omega = j_{K} j_{K-1} \dots j_{1}$,
\[
\int \prod_{k=1}^{K} 2 H_{j_{k}}(R^{k-1}(x,y)) \ d \lambda(y) = \prod_{k=1}^{K} 2 \int H_{j_{k}}(4^{k-1}x,y) \ d  \lambda(y).
\]
\end{lemma}

\begin{proof}
Let $F_{m}(x,y) = \prod_{k=m}^{K} 2 H_{j_{k}}(4^{k-1}x, 2^{k-m}y)$.  Note that
\[
F_{m}(x, \frac{y}{2}) = 2 H_{j_{m}}( 4^{m-1}x, \dfrac{y}{2}) \left( \prod_{k=m+1}^{K} 2 H_{j_{k}}(4^{k-1}x, 2^{k-(m+1)}y) \right)= 2 H_{j_{m}}( 4^{m-1}x, \dfrac{y}{2}) F_{m+1}(x, y).
\]
Likewise for $F_{m}(x, \frac{y+1}{2})$.

Since $\lambda$ is the invariant measure for the iterated function system $y \mapsto \frac{y}{2}$, $y \mapsto \frac{y+1}{2}$, we calculate:
\begin{align*}
\int_{0}^{1} F_{m}(x,y) \ d \lambda(y) &= \dfrac{1}{2} \left[ \int_{0}^{1} F_{m}(x, \frac{y}{2}) \ d \lambda(y) + \int_{0}^{1} F_{m}(x, \frac{y+1}{2}) \ d \lambda(y) \right] \\
&= \dfrac{1}{2} \left[ \int_{0}^{1} 2 H_{j_{m}}( 4^{m-1}x, \dfrac{y}{2}) F_{m+1}(x, y) + 2 H_{j_{m}}( 4^{m-1}x, \dfrac{y+1}{2}) F_{m+1}(x, y) \ d \lambda(y) \right] \\
&= \dfrac{1}{2} \left[ \int_{0}^{1} 2a_{j_{m},q} F_{m+1}(x, y) + 2a_{j_{m},q+2} F_{m+1}(x, y) \ d \lambda(y) \right] \\
&= \dfrac{1}{2} \left[ 2a_{j_{m},q} + 2a_{j_{m},q+2} \right]  \cdot \left[ \int_{0}^{1} F_{m+1}(x, y) \ d \lambda(y) \right] \\
&= \left[ \int_{0}^{1} 2 H_{j_{m}}(4^{m-1}x, y) \ d \lambda(y) \right] \cdot \left[ \int_{0}^{1} F_{m+1}(x, y) \ d \lambda(y) \right]
\end{align*}
where $q = 0$ if $0 \leq 4^{m-1} x < \frac{1}{2}$, and $q=1$ if $\frac{1}{2} \leq 4^{m-1} x < 1$.

The result now follows by a standard induction argument.
\end{proof}

\begin{proposition} \label{P:projection}  
Suppose the filters $m_{j}(x,y)$ are chosen so that 
\begin{enumerate}
\item[i)] the matrix $H$ is unitary,
\item[ii)] $a_{00} = a_{01} = a_{02} = a_{03} = \frac{1}{2}$, and 
\item[iii)] for $j=0,1,2,3$, $a_{j0} + a_{j2} = a_{j1} + a_{j3}$.
\end{enumerate}
Then for any word $\omega = j_{K} \dots j_{1}$,
\[ P_{V} S_{\omega} \mathds{1} = d_{\omega} e^{2 \pi i c(\omega) x}, \]
where 
\begin{equation} \label{Eq:coeff}
d_{\omega} = \prod_{k=1}^{K} \left( a_{j_{k}0} + a_{j_{k}2} \right).
\end{equation}
\end{proposition}

\begin{proof}
We apply the previous three Lemmas to obtain
\begin{align*}
[P_{V} S_{\omega} \mathds{1}](x,y) &= e^{2 \pi i c(\omega) x} \int \prod_{k=1}^{K} 2 H_{j_{k}}(4^{k-1}x,y) d \lambda(y) \\
&= e^{2 \pi i c(\omega) x} \prod_{k=1}^{K} 2 \int H_{j_{k}}(4^{k-1}x,y) d \lambda(y)
\end{align*}
By assumption iii), the integral $\int H_{j_{k}}(4^{k-1} x, y) d \lambda(y)$ is independent of $x$, and the value of the integral is $\frac{a_{j0}}{2} + \frac{a_{j2}}{2}$.  Equation \ref{Eq:coeff} now follows.
\end{proof}

\section{Concrete Constructions} \label{Sec:const}

We now turn to concrete constructions of Fourier frames for $\mu_{4}$.  The hypotheses of Lemma \ref{L:orthonormal} and Proposition \ref{P:projection} require $H$ to be unitary and requires the matrix
\[ A =
\left(
\begin{array} {rrrr}
a_{00} & a_{01} & a_{02} & a_{03} \\
a_{10} & a_{11} &  a_{12} & a_{13} \\
a_{20} & a_{21} &  a_{22} & a_{23} \\
a_{30} & a_{31} & a_{32} & a_{33}
\end{array}
\right)
\]
to have the first row be identically $\frac{1}{2}$ and to have the vector $\begin{pmatrix} 1 & -1 & 1 & -1 \end{pmatrix}^{T}$ in the kernel.

We can use Hadamard matrices to construct examples of such a matrix $A$.  Every $4 \times 4$ Hadamard matrix is a permutation of the following matrix:
\[ U_{\rho} =
\dfrac{1}{2}
\left(
\begin{array} {rrrr}
1 & 1 & 1 & 1 \\
1 & -1 & \rho & -\rho \\
1 & 1 & -1 & -1 \\
1 & -1 & -\rho  & \rho
\end{array}
\right)
\]
where $\rho$ is any complex number of modulus $1$.

If we set $H = U_{\rho}$, we obtain
\begin{equation} \label{Eq:Hrho}
A =
\dfrac{1}{2}
\left(
\begin{array} {rrrr}
1 & 1 & 1 & 1 \\
1 & 1 & \rho & \rho \\
1 & 1 & -1 & -1 \\
1 & 1 & -\rho  & -\rho
\end{array}
\right)
\end{equation}
which has the requisite properties to apply Lemma \ref{L:orthonormal} and Proposition \ref{P:projection}.

We define for $k = 1,2,3$, $l_{k} : \mathbb{N}_{0} \to \mathbb{N}_{0}$ by $l_{k}(n)$ is the number of digits equal to $k$ in the base $4$ expansion of $n$.  Note that $l_{k}(0) = 0$, and we follow the convention that $0^{0} = 1$.

\begin{theorem} \label{Th:main1}
For the choice $A$ as in Equation (\ref{Eq:Hrho}) with $\rho \neq -1$, the sequence
\begin{equation} \label{Eq:pars1}
\left\{ \left(\dfrac{1 + \rho}{2}\right)^{l_{1}(n)} 0^{l_{2}(n)} \left(\dfrac{1 - \rho}{2}\right)^{l_{3}(n)} e^{2 \pi i n x} : n \in \mathbb{N}_{0} \right\}
\end{equation}
is a Parseval frame in $L^2(\mu_{4})$.
\end{theorem}

\begin{proof}
By Lemma \ref{L:orthonormal}, we have that $ \{ S_{\omega} \mathds{1} : \omega \in X_{4} \}$ is an orthonormal set.  For a word $\omega = j_{K} j_{K-1} \dots j_{1}$, Proposition \ref{P:projection} yields that
\[
P_{V} S_{\omega} \mathds{1} =  e^{2 \pi i c(\omega) x} \prod_{k=1}^{K} \left(a_{j_{k}0} + a_{j_{k}2} \right).
\]
Then, setting $n = c(\omega)$, we obtain
\[ 
P_{V} S_{\omega} \mathds{1} = e^{2 \pi i n x} \left( a_{00} + a_{02} \right)^{K-l_{1}(n) - l_{2}(n) - l_{3}(n)}  \prod_{j=1}^{3} \left( a_{j0} + a_{j2} \right)^{l_{j}(n)}.
\]
Since
\[ a_{00} + a_{02} = 1, \quad a_{10} + a_{12} = \dfrac{ 1 + \rho }{2}, \quad a_{20} + a_{22} = 0, \quad a_{30} + a_{32} = \dfrac{ 1 - \rho }{2},\]
it follows that
\[
P_{V} S_{\omega} \mathds{1} = \left(\dfrac{1 + \rho}{2}\right)^{l_{1}(n)} 0^{l_{2}(n)} \left(\dfrac{1 - \rho}{2}\right)^{l_{3}(n)} e^{2 \pi i n x}.
\]
Since $c$ is a bijection, the set $\{ P_{V} S_{\omega} \mathds{1} : \omega \in X_{4} \}$ coincides with the set in (\ref{Eq:pars1}).

In order to establish that the set (\ref{Eq:pars1}) is a Parseval frame, we wish to apply Lemma \ref{L:frame-proj}, which requires that the subspace $V$ is contained in the closed span of $\{ S_{\omega} \mathds{1} : \omega \in X_{4} \}$.  Denote the closed span by $\mathcal{K}$.  We will proceed in a manner similar to Theorem \ref{diez-subsp}.  Define the function $f: \mathbb{R} \to V$ by $f(t) = e_{t}$ where $e_{t}(x,y) = e^{2 \pi i x t}$.  Note that $f(0) = \mathds{1} \in \mathcal{K}$.  Likewise, define a function $h_{X} : \mathbb{R} \to \mathbb{R}$ by
\[ h_{X}(t) = \sum_{\omega \in X_{4}} | \langle f(t) , S_{\omega} \mathds{1} \rangle |^2 = \| P_{\mathcal{K}} f(t) \|^2. \]

\begin{claim}  \label{Cl:constant}
We have $h_{X} \equiv 1$.
\end{claim}

Assuming for the moment that the claim holds, we deduce that $f(t) \in \mathcal{K}$ for every $t \in \mathbb{R}$.  Since $\{ f(\gamma) : \gamma \in \Gamma_{4} \}$ is an orthonormal basis for $V$, it follows that the closed span of $\{ f(t) : t \in \mathbb{R} \}$ is all of $V$.  We conclude that $V \subset \mathcal{K}$, and so Lemma \ref{L:frame-proj} implies that $\{ P_{V} S_{\omega} \mathds{1} : \omega \in X_{4} \}$ is a Parseval frame for $V$, from which the Theorem follows.

Thus, we turn to the proof of Claim \ref{Cl:constant}.  First, we require $\{ S_{\omega} \mathds{1} : \omega \in X_{4} \} = \cup_{j=0}^{3} \{ S_{j} S_{\omega} \mathds{1} : \omega \in X_{4} \}$, where the union is disjoint.  Clearly, the RHS is a subset of the LHS, and the union is disjoint.  Consider an element of the LHS: $S_{\omega} \mathds{1}$.  If $| \omega| \geq 2$, we write $S_{\omega} \mathds{1} = S_{j} S_{\omega_{0}} \mathds{1}$ for some $j$ and some $\omega_{0} \in X_{4}$, whence $S_{\omega} \mathds{1}$ is in the RHS.  If $ | \omega | = 1$, then we write $S_{\omega} \mathds{1} = S_{j} \mathds{1} = S_{j} S_{0} \mathds{1}$, which is again an element of the RHS.  Equality now follows.

As a consequence, 
\begin{align*}
h_{X}(t) &= \sum_{\omega \in X_{4}} | \langle f(t) , S_{\omega} \mathds{1} \rangle |^2 \\
&=  \sum_{j=0}^{3}  \sum_{\omega \in X_{4}} | \langle f(t) , S_{j} S_{\omega} \mathds{1} \rangle |^2 \\
&= \sum_{j=0}^{3}  \sum_{\omega \in X_{4}} | \langle S^{*}_{j} f(t) , S_{\omega} \mathds{1} \rangle |^2.
\end{align*}

We calculate:
{\allowdisplaybreaks 
\begin{align*}
[S_{j}^{*} f(t)](x,y) &= \dfrac{1}{2} \sum_{k=0}^{3} \overline{ m_{j}(\Upsilon_{k}(x,y)) } e_{t}(\Upsilon_{k}(x,y)) \\
&=  \dfrac{1}{2} \left[ \overline{a_{j0}} e^{-2 \pi i j x/4} e_{t}(\frac{x}{4}, \frac{y}{2}) + e^{-\pi i j} \overline{a_{j1}} e^{-2 \pi i j x/4} e_{t}(\frac{x+2}{4}, \frac{y}{2}) \right. \\
& \hspace{2cm} \left. + \overline{a_{j2}} e^{-2 \pi i j x/4} e_{t}(\frac{x}{4}, \frac{y+1}{2}) + e^{-\pi i j} \overline{a_{j3}} e^{-2 \pi i j x/4} e_{t}(\frac{x+2}{4}, \frac{y+1}{2})  \right] \\
&=  \dfrac{1}{2} \left[ \overline{a_{j0}} e^{-2 \pi i j x/4} e_{t}(\frac{x}{4}, \frac{y}{2}) + e^{-\pi i j} \overline{a_{j1}} e^{-2 \pi i j x/4} e^{\pi i t} e_{t}(\frac{x}{4}, \frac{y}{2}) \right. \\
& \hspace{2cm} \left. + \overline{a_{j2}} e^{-2 \pi i j x/4}  e_{t}(\frac{x}{4}, \frac{y}{2}) + e^{-\pi i j} \overline{a_{j3}} e^{-2 \pi i j x/4} e^{\pi i t}  e_{t}(\frac{x}{4}, \frac{y}{2})  \right] \\
&=  \dfrac{1}{2} \left[ \overline{a_{j0}}  + e^{-\pi i j} \overline{a_{j1}}  e^{\pi i t}  + \overline{a_{j2}} + e^{-\pi i j} \overline{a_{j3}}  e^{\pi i t}  \right] e^{-2 \pi i j x/4}  e_{t}(\frac{x}{4}, \frac{y}{2}) \\
&=  \dfrac{1}{2} \left[ \overline{a_{j0}}  + e^{-\pi i j} \overline{a_{j1}}  e^{\pi i t}  + \overline{a_{j2}} + e^{-\pi i j} \overline{a_{j3}}  e^{\pi i t}   \right] e^{2 \pi i ( t\frac{x}{4} - j \frac{x}{4})} \\
&=  \dfrac{1}{2} \left[ \overline{a_{j0}}  + e^{-\pi i j} \overline{a_{j1}}  e^{\pi i t}  + \overline{a_{j2}} + e^{-\pi i j} \overline{a_{j3}}  e^{\pi i t}  \right] e^{2 \pi i ( \frac{t-j}{4} x)} \\
&=  \dfrac{1}{2} \left[ \overline{a_{j0}}  + e^{-\pi i j} \overline{a_{j1}}  e^{\pi i t}  + \overline{a_{j2}} + e^{-\pi i j} \overline{a_{j3}}  e^{\pi i t}   \right] e_{\frac{t-j}{4}} (x,y).
\end{align*}
}

Thus, we define
\[ \mathfrak{m}_{j}(t) = \frac{1}{2} \left( \overline{a_{j0}} + \overline{a_{j2}} \right) + \frac{e^{- \pi i j}}{2} \left( \overline{a_{j1}} + \overline{a_{j3}} \right) e^{\pi i t}, \]
and
\[ g_{j}(t) = \frac{t-j}{4}. \]
As a consequence, we obtain
\begin{align}
h_{X}(t) &= \sum_{j=0}^{3}  \sum_{\omega \in X_{4}} | \langle S^{*}_{j} f(t) , S_{\omega} \mathds{1} \rangle |^2 \notag \\
&= \sum_{j=0}^{3} \sum_{\omega \in X_{4}} | \langle \mathfrak{m}_{j} (t) f(g_{j}(t)) , S_{\omega} \mathds{1} \rangle |^2  \notag \\
&= \sum_{j=0}^{3} | \mathfrak{m}_{j} (t) |^2 h_{X} (g_{j}(t)).  \label{Eq:ruelle}
\end{align}

Because of our choice of coefficients in the matrix $A$, which has the vector $\begin{pmatrix} 1 & -1 & 1 & -1 \end{pmatrix}^{T}$ in the kernel, we have for every $j$: $a_{j0} + a_{j2} = a_{j1} + a_{j3}$.  Thus, if we let $b_{j} = \overline{a_{j0}} + \overline{a_{j2}}$, the functions $\mathfrak{m}_{j}$ simplify to
\[ \mathfrak{m}_{j}(t) = b_{j} e^{\pi i \frac{t}{2} } \cos (\pi \frac{t}{2} ) \]
for $j=0,2,$ and
\[ \mathfrak{m}_{j}(t) = - i b_{j} e^{\pi i \frac{t}{2} } \sin (\pi \frac{t}{2} ) \] 
for $j = 1,3$.  Substituting these into Equation (\ref{Eq:ruelle}), 
\begin{equation} \label{Eq:ruelle1}
h_{X}(t) = \cos^2\left(\frac{\pi t}{2}\right) h_{X} \left(\frac{t}{4}\right)+
\sin^2\left(\frac{\pi t}{2}\right)\frac{|1+ \overline{\rho}|^2}{2} h_{X} \left(\frac{t-1}{4}\right)+
\sin^2\left(\frac{\pi t}{2}\right)\frac{|1- \overline{\rho}|^2}{2} h_{X} \left(\frac{t-3}{4}\right) .
\end{equation}

\begin{claim} \label{Cl:entire}
The function $h_{X}$ can be extended to an entire function.
\end{claim}

Assume for the moment that Claim \ref{Cl:entire} holds, we finish the proof of Claim \ref{Cl:constant}.  If $h_{X}(t) = 1$ for $t \in [-1,0]$, then $h_{X}(z) = 1$ for all $z \in \mathbb{C}$, and Claim \ref{Cl:constant} holds.  

Now, assume to the contrary that $h_{X}(t)$ is not identically $1$ on $[-1,0]$.  Since $0 \leq h_{X}(t) \leq 1$ for $t$ real,  then $\beta = \text{min} \{h_{X}(t) : t \in [-1,0] \} < 1$.  Because constant functions satisfy (\ref{Eq:ruelle1}), $h_1:=h_{X} - \beta$  also satisfies Equation (\ref{Eq:ruelle1}).  There exists $t_0$ such that $h_1(t_0)=0$ and $t_{0}\neq0$ as $h_{X}(0)=1$.  Since $h_1\geq 0$ each of the terms in (\ref{Eq:ruelle1}) must vanish :
\begin{equation}\label{r1}
\cos^2\left(\frac{\pi t_0}{2}\right) h_1\left(\frac{t_0}{4}\right)=0
\end{equation}
\begin{equation}\label{r2}
\sin^2\left(\frac{\pi t_0}{2}\right)\frac{|1+ \overline{\rho}|^2}{2}h_1\left(\frac{t_0-1}{4}\right)=0
\end{equation}
\begin{equation}\label{r3}
\sin^2\left(\frac{\pi t_0}{2}\right)\frac{|1- \overline{\rho}|^2}{2}h_1\left(\frac{t_0-3}{4}\right)=0
\end{equation}
Our hypothesis is that $\rho \neq -1$, so in Equation (\ref{r2}), the coefficient $\frac{| 1 + \overline{\rho} |}{2} \neq 0$.

Case 1: If $t_0\neq -1$ then Equation (\ref{r1}) implies $h_1(t_0/4)=0=h_1(g_0(t_0))$. Let $t_1:=g_0(t_0)\in (-1,0)$; iterating the previous argument implies that $h_{1}(g_0(t_{1})) = 0$. Thus, we obtain an infinite sequence of zeroes of $h_1$.

Case 2: If $t_0=-1$, then the previous argument does not hold.  However, we can construct another zero of $h_1$, $t_0'\in (-1,0)$ to which the previous argument will hold.  Indeed, if $t_{0} = -1$, Equation (\ref{r2}) implies $h_1((t_0-1)/4)=h_1(-1/2)=0$. Let $t_0'=-1/2$ and continue as in Case 1.  

In either case, $h_1$ vanishes on a (countable) set with an accumulation point, and since $h_1$ is analytic it follows that $h_1\equiv 0$, a contradiction, and Claim \ref{Cl:constant} holds.

Now, to prove Claim \ref{Cl:entire}, we follow the proof of Lemma 4.2 of \cite{JP98}.  For a fixed $\omega \in X_{4}$, define $f_{\omega}  : \mathbb{C} \to \mathbb{C}$ by
\[ f_{\omega} (z) = \langle e_{z} , S_{\omega} \mathds{1} \rangle = \int e^{2 \pi i z x} \overline{[S_{\omega} \mathds{1}] (x,y)} \ d(\mu_{4} \times \lambda). \]
Since the distribution $\overline{[S_{\omega} \mathds{1}] (x,y)} \ d(\mu_{4} \times \lambda)$ is compactly supported, a standard convergence argument demonstrates that $f_{\omega}$ is entire.  Likewise, $f^{*}_{\omega} (z) = \overline{ f_{\omega} (\overline{z})}$ is entire, and for $t$ real,
\[ f_{\omega}(t) f^{*}_{\omega}(t) = \left(\langle e_{t}, S_{\omega} \mathds{1} \rangle \right) \left( \overline{ \langle e_{t}, S_{\omega} \mathds{1} \rangle } \right) = | \langle e_{t}, S_{\omega} \mathds{1} \rangle |^2. \]
Thus,
\[ h_{X} (t) = \sum_{\omega \in X_{4} } f_{\omega} (t) f^{*}_{\omega}(t). \]
For $n \in \mathbb{N}$, let $h_{n}(z) = \sum_{ |\omega| \leq n } f_{\omega} (z) f^{*}_{\omega}(z)$, which is entire.  By H\"older's inequality, 
\begin{align*}
\sum_{\omega \in X_{4}} | f_{\omega} (z) f^{*}_{\omega}(z) | &\leq \left( \sum_{\omega \in X_{4}} | \langle e_{z} , S_{\omega} \mathds{1} \rangle |^2 \right)^{1/2} \left( \sum_{\omega \in X_{4}} | \langle e_{\overline{z}} , S_{\omega} \mathds{1} \rangle |^2 \right)^{1/2} \\
& \leq \| e_{z} \| \| e_{\overline{z}} \| \\
& \leq e^{K Im(z)}
\end{align*}
for some constant $K$.  Thus, the sequence $h_{n}(z)$ converges pointwise to a function $h(z)$, and are uniformly bounded on strips $Im(z) \leq C$.  By the theorems of Montel and Vitali, the limit function $h$ is entire, which coincides with $h_{X}$ for real $t$, and Claim \ref{Cl:entire} is proved.
\end{proof}

\begin{example} \label{Ex:incomplete}
As mentioned in Section \ref{Sec:ortho}, in general, $\{ S_{\omega} \mathds{1} \}$ need not be complete, and the exceptional point $\rho = -1$ in Theorem \ref{Th:main1} provides the example.  In the case $\rho = -1$, the set (\ref{Eq:pars1}) becomes
\[ \{ d_{n} e^{2 \pi i n x} : n \in \mathbb{N}_{0} \} \]
where the coefficients $d_{n} = 1$ if $n \in \Gamma_{3}$ and $0$ otherwise.  Here, 
\[ \Gamma_{3} = \{ \sum_{n=0}^{N} l_{n} 4^{n} : l_{n} \in \{0, 3\} \} \]
and it is known \cite{DHS09a} that the sequence $\{ e^{2 \pi i n x} : n \in \Gamma_{3} \}$ is incomplete in $L^2(\mu_{4})$.  Thus, $\{P_V S_{\omega} \mathds{1}\}$ is incomplete in $V$, so $\{ S_{\omega} \mathds{1} \}$ is incomplete in $L^2(\mu_{4} \times \lambda)$.
\end{example}

We can generalize the construction of Theorem \ref{Th:main1} as follows.  We want to choose a matrix
\[
A =
\begin{pmatrix}
\frac{1}{2} & \frac{1}{2} & \frac{1}{2} & \frac{1}{2} \\
h_{10} & h_{11} & h_{12} & h_{13} \\
h_{20} & h_{21} & h_{22} & h_{23} \\
h_{30} & h_{31} & h_{32} & h_{33}
\end{pmatrix}
\]
such that $\begin{pmatrix} 1 & -1 & 1 & -1 \end{pmatrix}^{T}$ is in the kernel of $H$ and the matrix
\[
H =
\left(
\begin{array} {rrrr}
\frac{1}{2} & \frac{1}{2} & \frac{1}{2} & \frac{1}{2} \\
h_{10} & -h_{11} & h_{12} & -h_{13} \\
h_{20} & h_{21} & h_{22} & h_{23} \\
h_{30} & -h_{31} & h_{32} & -h_{33}
\end{array}
\right)
\]
is unitary.  We obtain a system of nonlinear equations in the $12$ unknowns.  To parametrize all solutions, we consider the following row vectors:

\begin{align}
\vec{v}_{0} &= \frac{1}{2}\begin{pmatrix} 1 & 1 & 1 & 1 \end{pmatrix}  & \vec{w}_{0} &= \frac{1}{2}\begin{pmatrix} 1 & -1 & 1 & -1 \end{pmatrix} \\
\vec{v}_{1} &= \frac{1}{2}\begin{pmatrix}  1 & -1  & -1 & 1 \end{pmatrix}  & \vec{w}_{1} &= \frac{1}{2}\begin{pmatrix} 1 & 1 & -1 & -1 \end{pmatrix} \\
\vec{v}_{2} &= \frac{1}{2}\begin{pmatrix} 1 & 1 & -1 & -1 \end{pmatrix} & \vec{w}_{2} &= \frac{1}{2}\begin{pmatrix} 1 & -1 & -1 & 1 \end{pmatrix}
\end{align}

If we construct the matrix $A$ so that the rows are linear combinations of $\{ \vec{v}_{0}, \vec{v}_{1}, \vec{v}_{2} \}$, then $A$ will satisfy the desired condition on the kernel.  Note that if the $j$-th row of $A$ is $\alpha_{j0} \vec{v}_{0} + \alpha_{j1} \vec{v}_{1} + \alpha_{j2} \vec{v}_{2}$ for $j = 1,3$, then the $j$-th row of $H$ is $\alpha_{j0} \vec{w}_{0} + \alpha_{j1} \vec{w}_{1} + \alpha_{j2} \vec{w}_{2}$, whereas if $j=0,2$, then the $j$-th row of $H$ is equal to the $j$-th row of $A$.

Thus, we want to choose coefficients $\alpha_{jk}$, $j=0,1,2,3$, $k=1,2,3$ so that the matrix
\begin{equation} \label{Eq:matrixA}
H =
\begin{pmatrix}
\alpha_{00} \vec{v}_{0} + \alpha_{01} \vec{v}_{1} + \alpha_{02} \vec{v}_{2} \\
\alpha_{10} \vec{w}_{0} + \alpha_{11} \vec{w}_{1} + \alpha_{12} \vec{w}_{2} \\
\alpha_{20} \vec{v}_{0} + \alpha_{21} \vec{v}_{1} + \alpha_{22} \vec{v}_{2} \\
\alpha_{30} \vec{w}_{0} + \alpha_{31} \vec{w}_{1} + \alpha_{32} \vec{w}_{2}
\end{pmatrix}
\end{equation}
is unitary.  To satisfy the requirement on the first row, we choose $\alpha_{00} = 1$ and $\alpha_{01} = \alpha_{02} = 0$.
Calculating the inner products of the rows of $H$, we obtain the following necessary and sufficient conditions:
\begin{align}
| \alpha_{j0} |^2 + | \alpha_{j1} |^2 + | \alpha_{j2} |^2 &= 1  \label{Eq:ortho1} \\
\alpha_{00} \overline{\alpha_{20}}   &= 0 \label{Eq:ortho02} \\
\alpha_{11} \overline{\alpha_{22}}  + \alpha_{12} \overline{\alpha_{21}} &= 0 \label{Eq:ortho12} \\
\alpha_{10} \overline{\alpha_{30}}  + \alpha_{11} \overline{\alpha_{31}} + \alpha_{12} \overline{\alpha_{32}} &= 0 \label{Eq:ortho13} \\
\alpha_{21} \overline{\alpha_{32}}  + \alpha_{22} \overline{\alpha_{31}} &= 0 \label{Eq:ortho23}
\end{align}

\begin{proposition}  \label{P:solution1}
Fix $\alpha_{00} =1$.   There exists a solution to the Equations (\ref{Eq:ortho1}) - (\ref{Eq:ortho23}) if and only if $\alpha_{10}, \alpha_{30} \in \mathbb{C}$ with
\begin{equation} \label{Eq:duality}
|\alpha_{10}|^2 + |\alpha_{30}|^2 = 1.
\end{equation}
\end{proposition}

\begin{proof} ($\Leftarrow$)
If $|\alpha_{10}|^2 = 1$, then we choose $\alpha_{21} = \alpha_{31} = 1$ and all other coefficients to be $0$ to obtain a solution to Equations (\ref{Eq:ortho1}) - (\ref{Eq:ortho23}).  Likewise, if $| \alpha_{10}|^2 = 0$, then choose $\alpha_{11} = \alpha_{21} = 1$ and all other coefficients to be $0$.

Now suppose that $0 < |\alpha_{10}| < 1$, and we choose $\lambda = \dfrac{- \overline{\alpha_{10}} \alpha_{30}}{1 - |\alpha_{10}|^2}$.  Then choose $\alpha_{11}$ and $\alpha_{12}$ such that $|\alpha_{11}|^2 + |\alpha_{12}|^2 = 1 - |\alpha_{10}|^2$.  Now let $\alpha_{31} = \lambda \alpha_{11}$ and $\alpha_{32} = \lambda \alpha_{12}$.  We have
\begin{align}
\alpha_{10} \overline{\alpha_{30}}  + \alpha_{11} \overline{\alpha_{31}} + \alpha_{12} \overline{\alpha_{32}} &= \notag
\alpha_{10} \overline{\alpha_{30}}  + \overline{\lambda} |\alpha_{11}|^2  + \overline{\lambda} |\alpha_{12}|^2 \notag \\
&= \alpha_{10} \overline{\alpha_{30}}  + \overline{\lambda} (1 - | \alpha_{10}|^2) \label{Eq:lambda} \\
&= 0, \notag
\end{align}
so Equation (\ref{Eq:ortho13}) is satisfied.

Equation (\ref{Eq:ortho02}) forces $\alpha_{20} = 0$; choose $\alpha_{21}$ and $\alpha_{22}$ such that $|\alpha_{21}|^2 + |\alpha_{22}|^2 = 1$ and $\alpha_{11} \overline{\alpha_{21}} + \alpha_{12} \overline{\alpha_{22}} = 0$.  Thus, Equations (\ref{Eq:ortho12}) and (\ref{Eq:ortho23}) are satisfied.  Finally, regarding Equation (\ref{Eq:ortho1}), it is satisfied for $j=0,1,2$ by construction.  For $j=3$, we calculate:
\begin{align}
|\alpha_{30}|^2 + |\alpha_{31}|^2 + |\alpha_{32}|^2  &=  | \alpha_{30} |^2 + |\lambda|^2 \left( |\alpha_{11}|^2 + | \alpha_{12}|^2 \right)  \notag \\
&= |\alpha_{30}|^2+ \dfrac{|\alpha_{10}|^2 | \alpha_{30} |^2}{(1 - | \alpha_{10}|^2)^2}  \left( 1 - | \alpha_{10}|^2  \right)  \notag \\
&= | \alpha_{30}|^2 \left( 1 + \dfrac{ | \alpha_{10} |^2 }{ 1 - | \alpha_{10} |^2} \right) \notag \\
&= \dfrac{ | \alpha_{30} |^2 }{ 1 - | \alpha_{10} |^2 } \label{Eq:norm} \\
&= 1 \notag
\end{align}
as required.

($\Rightarrow$)  Suppose that we have a solution to Equations (\ref{Eq:ortho1}) - (\ref{Eq:ortho23}).  If $|\alpha_{10}| = 1$, then we must have $\alpha_{11} = \alpha_{12} = 0$, and thus Equation (\ref{Eq:ortho13}) requires $\alpha_{30} = 0$, so Equation (\ref{Eq:duality}) holds.

Now  suppose $|\alpha_{10}| < 1$.  Since $\alpha_{20} = 0$, we must have that $| \alpha_{21} |^2 + | \alpha_{22} |^2 = 1$.  Combining this with Equations (\ref{Eq:ortho12}) and (\ref{Eq:ortho23}) imply that the matrix 
\[ \begin{pmatrix} \alpha_{11} & \alpha_{12} \\ \alpha_{31}  & \alpha_{32} \end{pmatrix} \]
is singular.  Thus, there exists a $\lambda$ such that $\alpha_{31} = \lambda \alpha_{11}$ and $\alpha_{32} = \lambda \alpha_{12}$.  Using the same computation as in Equation (\ref{Eq:lambda}), we conclude that $\lambda = \dfrac{- \overline{\alpha_{10}} \alpha_{30}}{1 - |\alpha_{10}|^2}$; then Equation (\ref{Eq:norm}) implies (\ref{Eq:duality}).
\end{proof}

The coefficient matrix we obtain from this construction is
\[ H = \dfrac{1}{2}
\left(
\begin{array} {rrrr}
1 & 1 & 1 & 1 \\
\alpha_{10} + \alpha_{11} + \alpha_{12} & \alpha_{10} - \alpha_{11} + \alpha_{12} & \alpha_{10} - \alpha_{11} - \alpha_{12} & \alpha_{10} + \alpha_{11} - \alpha_{12} \\
\alpha_{21} + \alpha_{22} & \ - \alpha_{21} + \alpha_{22} &  \ - \alpha_{21} - \alpha_{22} &  \alpha_{21} - \alpha_{22} \\
\alpha_{30} + \lambda \alpha_{11} + \lambda \alpha_{12} & \alpha_{30} - \lambda \alpha_{11} + \lambda \alpha_{12} & \alpha_{30} - \lambda \alpha_{11} -  \lambda \alpha_{12} & \alpha_{30} + \lambda \alpha_{11} -  \lambda \alpha_{12}
\end{array}
\right)
\]
where we are allowed to choose $\alpha_{11}$, $\alpha_{12}$, $\alpha_{21}$ and $\alpha_{22}$ subject to the normalization condition in Equation (\ref{Eq:ortho1}).  However, those choices do not affect the construction, since if we apply Proposition \ref{P:projection} and the calculation from Theorem \ref{Th:main1}, we obtain
\begin{equation} \label{Eq:general}
P_{V} S_{\omega} \mathds{1} = (\alpha_{10})^{\ell_{1}(n)} \cdot (0)^{\ell_{2}(n))} \cdot (\alpha_{30})^{\ell_{3}(n)} e^{2 \pi i n x}.
\end{equation}
This will in fact be a Parseval frame for $L^2(\mu_{4})$, provided $V \subset \mathcal{K}$, as in the proof of Theorem \ref{Th:main1}.

\begin{theorem} \label{Th:main2}
Suppose $p, q \in \mathbb{C}$ with $|p|^2 + |q|^2 = 1$.  Then 
\[ 
\{ p^{\ell_{1}(n)} \cdot 0^{\ell_{2}(n)} \cdot q^{\ell_{3}(n)} e^{2 \pi i n x} : n \in \mathbb{N}_{0} \}
\]
is a Parseval frame for $L^2(\mu_{4})$, provided $p \neq 0$.
\end{theorem}

\begin{proof}
Substitute $\alpha_{10} = p$ and $\alpha_{30} = q$ in Proposition \ref{P:solution1} and Equation (\ref{Eq:general}).  As noted, we only need to verify $V \subset \mathcal{K}$.  We proceed as in the proof of Theorem \ref{Th:main1};  indeed,  define $f$, $h_{X}$, $\mathfrak{m}_{j}$ and $g_{j}$ as previously.  We obtain $b_{0} = 1$, $b_{1} = \overline{p}$, $b_{2} = 0$, and $b_{3} = \overline{q}$, so Equation (\ref{Eq:ruelle1}) becomes
\[
h_{X}(t) = \cos^2\left(\frac{\pi t}{2}\right) h_{X} \left(\frac{t}{4}\right)+
| \overline{p} |^2 \sin^2\left(\frac{\pi t}{2}\right) h_{X} \left(\frac{t-1}{4}\right)+
| \overline{q} |^2 \sin^2\left(\frac{\pi t}{2}\right) h_{X} \left(\frac{t-3}{4}\right) .
\]
From here, the same argument shows that $h_{X} \equiv 1$, and $V \subset \mathcal{K}$.
\end{proof}

\section{Concluding Remarks}

We remark here that the constructions given above for $\mu_{4}$ does not work for $\mu_{3}$.  Indeed, we have the following no-go result.  To obtain the measure $\mu_{3} \times \lambda$, we consider the iterated function system:
\[ \Upsilon_{0} (x,y) = ( \frac{x}{3}, \frac{y}{2} ), \ \Upsilon_{1} (x,y) = ( \frac{x+2}{3}, \frac{y}{2} ), \  \Upsilon_{2} (x,y) = ( \frac{x}{3}, \frac{y+1}{2} ), \  \Upsilon_{3} (x,y) = ( \frac{x+2}{3}, \frac{y+1}{2} ). \]
Using the same choice of filters, the matrix $\mathcal{M}(x,y)$ reduces to
\[ H =
\left(
\begin{array} {rrrr}
a_{00} & a_{01} & a_{02} & a_{03} \\
a_{10} & e^{4 \pi i/3} a_{11} &  a_{12} & e^{4 \pi i /3}a_{13} \\
a_{20} & e^{2 \pi i/3} a_{21} &  a_{22} & e^{2 \pi i/3} a_{23} \\
a_{30} & a_{31} & a_{32} & a_{33}
\end{array}
\right)
\]
which we require to be unitary.  Additionally, we require the same conditions as for $\mu_{4}$, namely, the first row of $H$ must have all entries $\frac{1}{2}$, and $a_{j0} + a_{j2} = a_{j1} + a_{j3}$.  The inner product of the first two rows must be $0$.  Hence,
\[ \dfrac{1}{2} \left( a_{10} + e^{4 \pi i/3} a_{11} + a_{12} + e^{4 \pi i /3}a_{13} \right) =  \dfrac{1}{2} \left( a_{10} + a_{12} \right) ( 1 + e^{4 \pi i/3} ) = 0. \]
Consequently, $a_{10} + a_{12} = 0$.  Likewise, $a_{20} + a_{22} = a_{30} + a_{32} = 0$.  As a result,
\[ H \begin{pmatrix} 1 \\ 0 \\ 1 \\ 0 \end{pmatrix} = \begin{pmatrix} a_{00} + a_{02} \\ a_{10} + a_{12} \\ a_{20} + a_{22} \\ a_{30} + a_{32} \end{pmatrix} = \begin{pmatrix} 1 \\ 0 \\ 0 \\ 0 \end{pmatrix} \]
and so $H$ cannot be unitary.

It may be possible to extend the construction for $\mu_{4}$ to $\mu_{3}$ by considering a representation of $\mathcal{O}_{n}$ for some sufficiently large $n$, or by considering $\mu_{3} \times \rho$ for some other fractal measure $\rho$ rather than $\lambda$.

\begin{acknowledgement}
We thank Dorin Dutkay for assisting with the proof of Claim 1 in Theorem \ref{Th:main1}.
\end{acknowledgement}

\nocite{LW02a,Str00a,Li07a}

\bibliographystyle{amsplain}
\bibliography{cuntz-4}

\end{document}